\documentclass[12pt,reqno]{amsart}

\usepackage{amsfonts,amssymb,amsmath,amsopn,amsthm,graphicx}
\usepackage{amsxtra, mathrsfs}
\usepackage{colordvi}
\usepackage[usenames,dvipsnames]{color}
\usepackage{amsfonts,amssymb,amsbsy,amsmath,amsthm,dsfont}
\usepackage[many]{tcolorbox}
\usetikzlibrary{decorations.pathreplacing}
\usepackage{lipsum}
\usepackage[many]{tcolorbox}
\usetikzlibrary{decorations.pathreplacing}

\newtcolorbox{leftbrace}{%
	enhanced jigsaw, 
	breakable, 
	frame hidden, 
	overlay={%
		\draw [
		decoration={brace,amplitude=0.5em},
		decorate,
		ultra thick,
		]
		(frame.south west)--(frame.north west);
	},
	parbox=false,
}

 \numberwithin{equation}{section}

\newtheorem{theorem}{Theorem}[section]
\newtheorem{corollary}[theorem]{Corollary}

\newtheorem{lemma}[theorem]{Lemma}
\newtheorem{proposition}[theorem]{Proposition}

\newtheorem{definition}[theorem]{Definition}

\newtheorem{remark}[theorem]{Remark}

\theoremstyle{definition}

\newtheorem{assumptions}[theorem]{Assumptions}

\usepackage{amssymb,amsmath,graphicx,amsthm,a4wide,wrapfig,caption,subcaption,epstopdf}
\usepackage[latin1]{inputenc}
\usepackage{enumitem}
\usepackage[colorlinks,pdfpagelabels,pdfstartview = FitH,bookmarksopen = true,bookmarksnumbered = true,linkcolor = blue,
plainpages = false,hypertexnames = false,citecolor = red,pagebackref=false]{hyperref}
\usepackage{verbatim}

\usepackage{lineno}
\usepackage{textcomp}
\usepackage{mathtools,hyperref}
\usepackage{cleveref}

\usepackage{setspace,esint}

\newcommand{\Hmm}[1]{\leavevmode{\marginpar{\tiny%
			$\hbox to 0mm{\hspace*{-0.5mm}$\leftarrow$\hss}%
			\vcenter{\vrule depth 0.1mm height 0.1mm width \the\marginparwidth}%
			\hbox to
			0mm{\hss$\rightarrow$\hspace*{-0.5mm}}$\\\relax\raggedright #1}}}

\newcommand{\loc}{{\rm loc}}

\newcommand{\N}{\mathbb{N}}

\newcommand{\R}{\mathbb{R}}

\theoremstyle{definition}

\numberwithin{equation}{section}
\newcommand{\diver}{\mathrm{div}\,}
\newcommand{\RN}[1]{%
	\textup{\uppercase\expandafter{\romannumeral#1}}%
}
\newcommand{\dx}{\,\mathrm{d}x}

\newcommand{\dt}{\,\mathrm{d}t}

\newcommand{\ds}{\,\mathrm{d}s}

\newcommand{\dtau}{\,\mathrm{d}\tau}

\newcommand{\dH}{\,\mathrm{d}\mathcal{H}^{n-1}}

\newcommand{\core}{C_0^{\infty}(\Omega)}

\newcommand{\be}{\begin{equation}}
	\newcommand{\ee}{\end{equation}}
\newcommand{\bea}{\begin{eqnarray}}
	\newcommand{\eea}{\end{eqnarray}}
\newcommand{\bean}{\begin{eqnarray*}}
	\newcommand{\eean}{\end{eqnarray*}}

\newcommand{\opname}[1]{\mbox{\rm #1}\,}
\newcommand{\supp}{\opname{supp}}

\newlength{\wex}  \newlength{\hex}
\newcommand{\ass}[1]{Let Assumption \ref{assump1} holds  in a bounded Lipschitz domain $\Gw$}

\def\ga{\alpha}     \def\gb{\beta}       
             
                         \def\vge{\varepsilon}
       \def\vgf{\varphi}    
            \def\gl{\lambda}

      \def\gw{\omega}

\def\Gw{\Omega}

\begin{document}

\title{Optimal Hardy-weights for the $(p,A)$-Laplacian with a potential term}

\author {Idan Versano}

\address {Idan Versano, Department of Mathematics, Technion - Israel Institute of
	Technology,   Haifa, Israel}

\email {idanv@campus.technion.ac.il}

\begin{abstract}
	We construct new optimal $L^p$ Hardy-type inequalities for elliptic  Schr\"odinger-type operators 
with a potential term.


	\noindent  2000  \! {\em Mathematics  Subject  Classification.} Primary 35A23; Secondary  35B09, 35J08, 35J25 \\[1mm]
\end{abstract}
\maketitle

\section{Introduction}

For any $\xi\in \R^n$ and a positive definite matrix $A\in \R^{n \times n}$, let  
$| \xi|_A:=\sqrt{\langle A\xi,\xi \rangle }$, where $\langle \cdot, \cdot \rangle$ denotes the Euclidean inner product on $\R^n$.
Consider a second order half-linear operator of the form 
$$
Q_{p,A,V}(u):=-\diver\left (|\nabla u|^{p-2}_{A}A\nabla u \right)+V|u|^{p-2}u
$$ defined in a domain $\Gw\subset \R^n$, and assume that $Q_{p,A,V}$ admits a positive solution in $\Gw$.
We are interested to find an optimal  weight function $W\gneqq 0$ (see Definition \ref{def:optimal}) such that the equation $Q_{p,A,V-W}(u)=0$ admits a positive solution in $\Gw$.
Equivalently \cite[Theorem 4.3]{PP}, we are interested to find an optimal  weight function $W\gneqq 0$  such that the following Hardy-type inequality is satisfied:
\begin{equation}\label{eq:first_func}
\int_{\Gw}(|\nabla \phi|_A^p+V|\phi|^p ) \dx\geq \int_{\Gw} W|\phi|^p \dx \qquad \forall \phi\in C_0^{\infty}(\Gw).
\end{equation}
In some definite sense, an optimal weight $W\gneqq 0$ is  "as large as possible" nonnegative function such that \eqref{eq:first_func} is satisfied for all nonnegative
$\phi\in \core$.

The search for Hardy-type inequalities with optimal weight
function $W$ was originally  proposed by Agmon, who   raised this problem in connection with
his theory of exponential decay of Schr\"odinger eigenfunctions \cite[p. 6]{Agmon}.
In the past four decades, the problem of improving Hardy-type inequalities has engaged many authors. In particular, 
Hardy-type inequalities were established for a vast class of operators (e.g., elliptic operators, Schr\"odinger operators on graphs, fractional differential equations) with different types of boundary conditions, see \cite{BEL,BFT,BD,BM,DFP,DP,EHL,KPP,PV}. 
In \cite{DP}, Devyver and Pinchover studied the problem of {\em optimal}  weights for the operator
$Q_{p,A,V}$. However, they managed to find optimal weights only in the case where $A$ is the identity matrix and $V=0$.  They proved (under certain assumptions ) 
that the $p$-Laplace operator, $-\diver\left (|\nabla u|^{p-2}\nabla u \right)$, admits an optimal Hardy-weight. 
More specifically, it is proved that if $1< p\leq n$, then 
$W= \left( \frac{p-1}{p}\right)^p\left| \frac{\nabla G}{G}\right|^p$ an optimal Hardy-weight, 
where $G$ is the associated positive minimal Green function with singularity at $0$. For $p>n$, several cases should be considered, depending on the behavior of a positive $p$-harmonic function with singularity at $0$.

In the present paper we make a nontrivial progress towards the study of \eqref{eq:first_func} in the case where $A$ is not necessarily the identity matrix, and $V$ is a slowly growing potential function.
Our main result reads as follows.
\begin{theorem}\label{thm:main_non_2}
	Let $\Gw\subset \R^n$ be a domain and  $x_0\in  \Gw$. Let
	$Q_{p,A,V}$ be a subcritical operator in $\Gw$ satisfying Assumptions \ref{assump2_non_linear}  in $\Gw$.   Suppose that $Q_{p,A,V}$ admits a (nonnegative) Green potential, $G_{\vgf}(x)$, in $\Gw$ (see Definition \ref{def:Green}) satisfying 
	\begin{equation*}
	\lim\limits_{x\to \overline{\infty}}G_{\vgf}(x)=0; \quad
	\int_{\Gw }V G_{\vgf}(x)^{p-1}\dx< 0;
	\quad
	\int_{\Gw }|V|G_{\vgf}(x)^{p-1}\dx<\infty,
	\end{equation*}
	where $\overline{\infty}$ denotes the ideal point in the one-point compactification of $\Gw$. 
	Then the operator $Q_{p,A,{V}/{c_p}}$ admits an optimal Hardy-weight in $\Gw$, where $c_p=\big (p/(p-1)\big )^{p-1}$.
\end{theorem}
As a corollary  of the proof of Theorem \ref{thm:main_non_2} we obtain the following result.
\begin{corollary}\label{thm:main_non_1}
	Let $\Gw\subset \R^n$ be a domain and  $x_0\in K\Subset \Gw$. Let
	$Q_{p,A,V}$ be a subcritical operator in $\Gw$ satisfying Assumptions \ref{assump2_non_linear} with $V\leq 0$ in $\Gw$.   Suppose that $Q_{p,A,V}$ admits a positive minimal Green function $G(x)$ in $\Gw\setminus \{x_0\}$ (see Definition \ref{def:Green}) satisfying 
	\begin{equation}\label{eq:Gphi_min_non}
	\lim\limits_{x\to \overline{\infty}}G(x)=0, \quad \mbox{ and }
	\int_{\Gw \setminus K}|V||G(x)|^{p-1}\dx<\infty,
	\end{equation}
	where $\overline{\infty}$ denotes the ideal point in the one-point compactification of $\Gw$. 
	Then the operator $Q_{p,A,{V}/{c_p}}$ admits an optimal Hardy-weight in $\Gw$, where $c_p=\big (p/(p-1)\big )^{p-1}$.
\end{corollary}

The paper is organized as follows. In Section~\ref{sec2}, we introduce the necessary notation and recall some previously obtained results needed in the present paper. We proceed in Section~\ref{sec3}, with proving essential results needed for the proof of Theorem \ref{thm:main_non_2}, and then we prove Theorem~\ref{thm:main_non_2} and Corollary~\ref{thm:main_non_1}.

\section{Preliminaries}\label{sec2}
Let $\Gw\subset \R^n$ be a domain, and let $1<p<\infty$.
Throughout
the paper we use the following notation and conventions:
\begin{itemize}
	\item For any $R>0$ and $x\in \R^n$, we denote by $B_R(x)$ the open ball of radius $R$ centered at $x$, and $B_R^{+}(0)=\{x\in B_R(0) : x_n >  0\}$.

	
	\item We write $\Gw_1 \Subset \Gw_2$ if $\Gw_2$ is open in $\Gw,$ the set $\overline{\Gw_1}$ is compact, and $\overline{\Gw_1}\subset \Gw_2$. 
	\item $C$ refers to a positive constant which may vary from  line to line.
	\item Let $g_1,g_2$ be two positive functions defined in $\Gw$. We use the notation $g_1\asymp g_2$ in
	$\Gw$ if there exists a positive constant $C$ such
	that
	$$C^{-1}g_{2}(x)\leq g_{1}(x) \leq Cg_{2}(x) \qquad \mbox{ for all } x\in \Gw.$$	
	
	\item Let $g_1,g_2$ be two positive functions defined in $\Gw$, and let $x_0\in \Gw$. We use the notation $g_1\sim g_2$ near $x_0$ if there exists a positive constant $C$ such that
	$$
	\lim_{x\to x_0}\frac{g_1(x)}{g_2(x)}=C.
	$$

	\item The gradient of a function $f$ will be denoted either by $\nabla f$ or $Df$.

	\item $\chi_B$ denotes the characteristic function of a set $B\subset \R^n$.

	\item For any $1\leq p\leq \infty$, $p'$ is the H\"older conjugate exponent of $p$ satisfying $p'=p/(p-1)$.
	\item For $1\leq p<n$,  $p^*:=np/(n-p)$ is its Sobolev critical exponent.
	%

	\item For a real valued function $W$, we write $W\gneqq 0$ in $\Gw$ if $W\geq 0$ in $\Gw$ and $\sup\limits_{\Gw}W>0$.
	
	
	\item For a symmetric positive definite  $A\in L_\loc^\infty(\Gw,\R^{n\times n})$, we denote $\Delta_{p,A} (u):=\diver(|\nabla u|_A^{p-2}A\nabla u)$ is  the $(p,A)$-Laplace operator. 
	
	\item For a real valued function $u$ and $1<p<\infty$, $\mathcal{I}_p(u):=|u|^{p-2}u$.
	
	\item  $\overline{\infty}$ denotes the ideal point in the one-point compactification of $\Gw$. 
	\item $\R_+$ denotes the segment $(0,\infty)$.
	\item $d_{\Gw}=\mathrm{dist}(\cdot,\partial \Gw):\Gw\to (0,\infty)$ is the distance function to $\partial \Omega$.   
	\item $\mathrm{diam}(\Gw)$ denotes the diameter of $\Gw$.
	\item $\mathrm{supp} (u)$ denotes the support of the function $u$.
	\item  $\mathcal{H}^l$, $1\leq l \leq n$, denotes the $l$-dimensional Hausdorff measure on $\R^n$. 
\end{itemize}
\subsection{Gauss-Green formula}
We continue with several definitions and results concerning the Gauss-Green theorem \cite{CTZ}. 
\begin{definition}
	Let $D\subset \R^n$ be an open set.
	\begin{enumerate}
		\item   We denote by $\mathcal{M}(D)$ the space of all signed Radon measures $\mu$ on $D$ such that $\int_{D}\,\text{d}|\mu|<\infty$.
		\item A vector field $F\in L^\infty(D,\R^n)$ is called {\em a divergence measure field}, written as $F\in \mathcal{DM}^{\infty}(D)$, if $\diver(F)=\mu\in \mathcal{M}(D)$, i.e., there exists $\mu\in \mathcal{M}(D)$ such that
		$$
		\int_D \phi \,\mathrm{d}\mu=-\int_{D}\nabla \phi \cdot F \dx \qquad \forall\, \phi\in C_0^{\infty}(D).
		$$
		\item We say that a vector field $F\in L^{\infty}_{\loc}(D,\R^n)$ belongs to $ \mathcal{DM}^\infty_{\loc}(D)$ if for any  open subset  $E\Subset D$, we have  $F\in \mathcal{DM}^{\infty}(E)$.
	\end{enumerate}
\end{definition}
\begin{definition}[{cf. \cite{CTZ} and \cite[Section 5]{EG}  }]
	Let $D\subset \R^n$ be an open set.
	A function $f\in L^1(D)$ has a {\em  bounded variation} in $D$ if
	$$
	\sup\left \{ \int_{D} f\, \diver(\phi) \dx: \phi\in C_0^1(D,\R^n), |\phi|\leq 1 \right \}< \infty.
	$$
	Denote by  $\mathrm{BV}(D)$  the space of all functions $f\in L^1(D)$ having bounded variation.
\end{definition}
\begin{definition}
	Let $D\subset \R^n$ be an open set.
	A measurable subset $E\subset \R^n$ is said to be  {\em a set of finite perimeter} in $D$ if $
	\chi_D\in \mathrm{BV}(D)$.
\end{definition}
\begin{proposition}[{\cite[Theorem 5.9, p. 212]{EG}\label{prop:coarea_f_peri}}]
	Let $E\Subset \R^n$ and let $0\leq f\in \mathrm{BV}(E)\cap C^1(E)$. Then, for a.e. $t\in [0,\infty)$ the set 
	$\{x\in E: f(x)>t\}$ has finite perimeter.
	In particular, for a.e. $0\leq t_1<t_2$
	the set 
	$\{x\in E: t_1<f(x)<t_2\}$ has finite perimeter.
\end{proposition}
We proceed with the following Gauss-Green theorem of divergence measure fields over sets of finite perimeter (see \cite[Theorems 5.2 and 7.2]{CTZ} and \cite[Proposition 3.1]{DP}).
\begin{lemma}\label{lem:Gauss_Green_non}
	Let $D\subset \R^n$ be an open set.  Suppose that
	$F\in \mathcal{DM}^{\infty}_{\loc}(D)$ with $\diver(F)=\mu\in \mathcal{M}(D)$. Let $E\Subset D$ be a set of finite perimeter satisfying
	\begin{itemize}
		\item $\partial E=\left(\bigcup\limits_{k\in \N}D_k\right) \cup N$,
		\item for each $k\in \N$, $D_k $ is $(n-1)$- dimensional $C^{1}$ surface, and $\mathcal{H}^{n-1}(N)=0$.
	\end{itemize}
	Then,  
	$$
	\int_{E}\diver(F)  \dx=\int_{\partial E}F\cdot \vec{n}\dH,
	$$
	where $\vec{n}$ is a classical outer unit normal to $\partial E$ which is defined $\mathcal{H}^{n-1}$-a.e.    on $\partial E$.
\end{lemma}
\subsection{Local Morrey spaces}
In the present subsection we introduce a certain class of Morrey spaces that depend on the index $p$, where $1<p<\infty$.
\begin{definition}
	Let $q\in [1,\infty]$ and $\omega \Subset \R^n$.
	For a measurable, real valued function $f$ defined in $\omega$, we set 
	$$
	\|f\|_{M^{q}(\omega)}:=		\sup_{\substack{y\in \omega\\ r<\mathrm{diam}(\omega)}} 
	\frac{1}{r^{n/q'}}\int_{\omega\cap B_r(y)}|f| \dx.
	$$
	We write $f\in M^{q}_{\loc}(\Gw)$ if for any $\omega \Subset \Gw$ we have $\|f\|_{M^{q}(\omega)}<\infty$.
\end{definition}
Next, we define a special local Morrey space $M^{q}_{\loc}(p;\Gw)$ which depends on the values of the exponent $p$.
\begin{definition}
	For $p\neq n$, we define 
	$$
	M^{q}_{\loc}(p;\Gw):=
	\begin{cases}
	M^{q}_{\loc}(\Gw) \mbox{~with~} q>n/p & \mbox{~if~} p<n \\
	L^1_{\loc}(\Gw) & \mbox{~if~} p>n,
	\end{cases}
	$$
	while for $p=n$, $f\in M^{q}_{\loc}(n;\Gw)$ means that for some $q>n$ and any $\omega \Subset \Gw$ we have
	$$
	\|f\|_{M^q_{n;\Gw}}:=	\sup_{\substack{y\in \omega\\ r<\mathrm{diam}(\omega)}} 
	\vgf_q(r)\int_{\omega\cap B_r(y)}|f| \dx<\infty,
	$$
	where $\vgf_q(r):=\log(\mathrm{diam}(\omega)/r)^{q/n'}$ and $0<r<\mathrm{diam}(\omega)$.
\end{definition}
For the regularity theory of equations with coefficients in Morrey spaces we refer the reader to  \cite{MZ,PP}.

We associate to any  domain  $\Gw\subset \R^n$ an {\em  exhaustion}, i.e. a sequence of smooth, precompact domains $\{\Gw_{j}\}_{j=1}^{\infty}$
such that $\Gw_1\neq \emptyset$,
$\Gw_j \Subset \Gw_{j+1}$ and $\bigcup_{j=1}^{\infty}\Gw_j=\Gw$.
\subsection{Criticality theory for $Q_{p,A,V}$}
Let $1<p<\infty$, and consider the operator
\begin{equation}
Q_{p,A,V}(u) :=-\Delta_{p,A}(u)+V\mathcal{I}_p(u), 
\end{equation}
defined on a domain
$\Gw\subset \R^n, n\geq 2$,
where  $\Delta_{p,A}:=\diver(| \nabla u|_A^{p-2} A  \nabla u)$ and $\mathcal{I}_p(u):=|u|^{p-2}u$.
Unless otherwise stated, we always assume that the matrix $A$ and the potential function $V$ satisfy the following regularity assumptions:
\begin{assumptions}
	\label{assump2_non_linear} 
	\begin{itemize}
		\item[{\ }]				
		\item $A(x)\!=\!(a^{ij}(x))_{i,j=1}^{n}\in C^{\alpha}_{\loc}(\Gw,\R^{n^2})$ is a symmetric positive definite  matrix which is locally
		uniformly elliptic, that is, for any compact $K\Subset \Gw$ there exists  $\Theta_K>0$ such that 
		\begin{eqnarray*} 
			\Theta_K^{-1}\sum_{i=1}^n\xi_i^2\leq\sum_{i,j=1}^n
			a^{ij}(x)\xi_i\xi_j\leq \Theta_K\sum_{i=1}^n\xi_i^2 \quad \forall \xi\in \mathbb{R}^n \mbox{ and } \forall x\in K.
		\end{eqnarray*}		
		\item $V\in M^{q}_{\loc}(p;\Gw)$ is a real valued function.	
	\end{itemize}
\end{assumptions}

The associated  energy functional   for the operator $Q_{p,A,V}$ in $\Gw$ is defined by
$$
\mathcal{Q}_{p,A,V}^{\Gw}(\phi) :=\int_{\Gw}(|\nabla \phi|_A^p+V|\phi|^p) \dx \qquad  \phi\in C_0^{\infty}(\Gw).
$$
\begin{definition}
	We say that $u\in W^{1,p}_{\loc}(\Gw)$ is a (weak)  {\em solution} (resp. {\em supersolution}) of $Q_{p,A,V}(u) =0$ in $\Gw$ if for any $\phi\in C_0^{\infty}(\Gw)$ (resp. $ 0\leq \phi\in C_0^{\infty}(\Gw) $)
	$$
	\int_{\Gw} |\nabla u|_{A}^{p-2}A\nabla u \cdot \nabla \phi \dx +
	\int_{\Gw}V|u|^{p-2}u\phi \dx = 0 \; (\mathrm{resp.} \geq 0 ).
	$$ 
\end{definition}
It should be noted that the above definition makes sense due to the
following Morrey-Adams Theorem (see for example, \cite[Theorem~2.4]{PP} and references therein).
\begin{theorem}[Morrey-Adams theorem]
	Let $\omega\Subset \R^n$ and $V\in M^q(p;\omega)$.
	\begin{enumerate}
		\item There exists a constant $C=C(n,p,q)>0$ such that for any $\delta>0$
		$$
		\int_{\omega}|V||u|^p \dx \leq \delta \|\nabla u\|^{p}_{L^p(\omega,\R^n)}+\
		\frac{C}{\delta^{n/(pq-n)}}\|V\|^{pq/(pq-n)}_{M^q(p;\omega)}\|u\|^{p}_{L^p(\omega)}
		\;\;  \forall u\in W^{1,p}_{0}(\omega).
		$$
		\item For any $\omega'\Subset \omega$ with Lipschitz boundary, there exists a positive constant $C= C(n,p,q,\omega',\omega,\delta,\|V\|_{M^q(p;\omega)})$ and $\delta_0$ such that for $0<\delta\leq \delta_0$
		$$
		\int_{\omega'}|V||u|^p \dx \leq \delta  \|\nabla u\|^{p}_{L^p(\omega',\R^n)}
		+C\|u\|^{p}_{L^p(\omega')} \quad  \forall u\in W^{1,p}(\omega').
		$$
	\end{enumerate}
\end{theorem}

We denote the set of all positive solutions (resp., supersolutions) of $Q_{p,A,V}(u)=0$ in $\Gw$ by $\mathcal{C}^{Q_{p,A,V}}(\Gw)$ (resp., $\mathcal{K}^{Q_{p,A,V}}(\Gw)).$
We say that the operator $Q_{p,A,V}$ is {\em nonnegative } (in short $Q_{p,A,V} \geq 0$)  in $\Gw$ if 
$\mathcal{C}^{Q_{p,A,V}}\neq \emptyset$.

\begin{remark}
	A weak (super)solution of the equation $-\Delta_{p,A}(u)=0$ in $\Gw$ is said to be a  {\em $(p,A)$-(super)harmonic function} in $\Gw$.	
\end{remark}
It is well known that  under Assumptions~\ref{assump2_non_linear} any positive solution of the equation $Q_{p,A,V}(u)=0$ in $\Gw$ belongs to $C^{1,\alpha}(\Gw)$ (see for example \cite[Remark 1.1]{PP}).
Furthermore,  the following 
Harnack convergence principle holds true.
\begin{proposition}[Harnack convergence principle { \cite[Proposition 2.7]{GP1}}]
	Let $\{\Gw_{k}\}_{k\in \N}$ be an exhaustion of $\Gw$.
	Assume that $\{A_k\}_{k\in \N}$ is a sequence of symmetric and locally uniformly positive definite  matrices 
	such that the local ellipticity constants does not depend on $k$, and $\{A_k\}_{k\in \N}\subset L^{\infty}_{\loc}(\Gw_k,\R^{n^2})$ 
	converges  weakly in $L^{\infty}_{\loc}(\Gw,\R^{n^2})$ to a matrix  $A\in L^{\infty}_{\loc}(\Gw,\R^{n^2})$. Assume further that $\{V_k\}_{k\in \N}\subset M^{q}_{\loc}(p;\Gw_k)$ converges  weakly  in $M^{q}_{\loc}(p;\Gw)$ to $V\in M^{q}_{\loc}(p;\Gw)$.
	For each $k$, let $v_k$ be a positive solution of the equation $Q_{p,A_k,V_k}(u)=0$ in $\Gw_k$ such that $v_k(x_0)=1$, where $x_0$ is a fixed reference point in $\Gw_1$.
	Then there exists $0<\eta<1$ such that, up to a subsequence, $\{v_k\}_{k\in \N}$   converges weakly in $W^{1,p}_{\loc}(\Gw)$ and in $C_{\loc}^\gb(\Gw)$ to a positive weak solution $v$ of the equation $Q_{p,A,V}(u)=0$ in $\Gw$ (here $0<\gb<\ga$).
\end{proposition}
\begin{definition}
	Let $\Gw\subset \R^n$ be a bounded Lipschitz domain.
	A {\em principal eigenvalue}  of $Q_{p,A,V}$ in $\Gw$ is an eigenvalue $\lambda$ of the problem
	$$
	\begin{cases}
	Q_{p,A,V}(u)=\lambda |u|^{p-2}u & \mbox{in~} \Gw, \\
	u=0 & \mbox{on~} \partial \Gw,
	\end{cases}
	$$
	with a  nonzero  nonnegative $u$ which is called  a {\em principal eigenfunction}.
\end{definition}
\begin{proposition}[{\cite[Theorem 3.9]{PP}}]
	Let $\Gw\subset \R^n$ be a bounded Lipschitz  domain, and assume that $A$ is a uniformly elliptic, bounded matrix in $\Gw$, and $V\in M^{q}(p;\Gw)$.
	Then, the operator $Q_{p,A,V}$ admits a unique  principal eigenvalue $\lambda_1(\Gw)$.  Moreover, $\gl_1$ is simple and its principal eigenfunction is the minimizer of the Rayleigh-Ritz
	variational problem  
	$$
	\lambda_1(\Gw)=\min\limits_{u \in W^{1,p}_{0}\setminus \{0\}}\dfrac{\mathcal{Q}^{\Gw}_{p,A,V}(u)}{\|u\|^p_{L^p(\Gw)}}\, .
	$$
\end{proposition}
The following well-known Allegretto-Piepenbrink theorem (in short, the {\bf AP} theorem) connects between the nonnegativity of $Q_{p,A,V}$ and the nonnegativity of its associated  energy functional $\mathcal{Q}^{\Gw}_{p,A,V}$ \cite[Theorem 4.3]{PP}.
\begin{theorem}[{\bf AP} theorem]
	The following assertions are equivalent.
	\begin{enumerate}
		\item $\mathcal{Q}_{p,A,V}^{\Gw}(\phi) \geq 0$ for all $\phi\in \core.$
		\item $\mathcal{C}^{Q_{p,A,V}}(\Gw)\neq \emptyset.$
		\item $\mathcal{K}^{Q_{p,A,V}}(\Gw)\neq \emptyset.$
	\end{enumerate}
\end{theorem}
\begin{definition}
	Assume that $Q_{p,A,V}\geq 0$ in $\Gw.$ We say that $Q_{p,A,V}$ is {\em subcritical} in $\Gw$ if there exists $0\lneqq W\in M^{q}_{\loc}(p;\Gw)$ such that $Q_{p,A,V-W}\geq 0$  in $\Gw$. 			We say that $Q_{p,A,V}$ is {\em critical} in $\Gw$ if for all $0\lneqq W\in M^{q}_{\loc}(p;\Gw)$ the equation 
	$Q_{p,A,V-W}(u)=0$ does not admit a positive solution in $\Gw.$
\end{definition}
\begin{definition}
	Let $\omega$ be a bounded Lipschitz domain.  We say that $Q_{p,A,V}$ satisfies the  {\em (generalized) weak maximum principle}  in $\gw$ if for any $u\in W^{1,p}(\omega)$ satisfying  $Q_{p,A,V}(u)\geq 0$ in $\gw$   and $u\geq0$ on $\partial \omega$, we have  $u\geq0$ in $\omega$.
	
	We say that  $Q_{p,A,V}$ satisfies the  {\em strong  maximum principle} in $\omega$  if for any $u\in W^{1,p}(\omega)$ satisfying  $Q_{p,A,V}(u)\!\geq\! 0$ in $\gw$   and $u\!\geq\!0$ on $\partial \omega$,  either $u \!=\! 0$, or  $u\!>\!0$ in $\omega$.
\end{definition}
\begin{lemma}[{\cite[Theorem 3.10]{PP}\label{lem:see_PP}}]
	Let $\Gw$ be a bounded Lipschitz domain, and assume that $A$ is
	a uniformly elliptic, bounded matrix in $\Gw$, and
	$V\in M^{q}(p;\Gw)$.  Then  the following assertions are equivalent.
	\begin{enumerate}
		\item $Q_{p,A,V}$ satisfies the (generalized) weak maximum principle in $\Gw$. 
		\item $Q_{p,A,V}$ satisfies the strong maximum principle in $\Gw$.
		\item The equation $Q_{p,A,V}(u)=0$ admits a positive supersolution in $W^{1,p}_{0}(\Gw)$ which is not a solution.
		\item The equation $Q_{p,A,V}(u)=0$ admits a positive supersolution in $W^{1,p}(\Gw)$ which is not a solution.
		\item $\lambda_1(\Gw)>0$.
		\item For any $0\leq g\in L^{p'}(\Gw)$, there exists a unique nonnegative solution in $W^{1,p}_{0}(\Gw)$ of $Q_{p,A,V}(u)=g$.
		
	\end{enumerate}
\end{lemma}
\begin{corollary}
	If there exists a weak  positive  (super)solution of $Q_{p,A,V} (u) =0$ in a domain $\Gw\subset \R^n$,  then 
	$\lambda_1(\Gw')>0$ for any bounded Lipschitz subdomain $\Gw'\Subset \Gw$.
\end{corollary}
\begin{definition}
	Let $K_0$  be a compact subset of $\Gw$. A positive solution $u$ of $Q_{p,A,V}(u)=0$ in $\Gw \setminus K_0$ is said to be a {\em positive
		solution of minimal growth in a neighborhood of infinity in $\Gw$}, and denoted by $u\in \mathcal{MG}_{A,V,\Gw;K_0}$, if for any smooth
	compact subset $K$ of $\Gw$ with $K_0 \Subset \mathrm{int} (K)$, and any positive  supersolution $v \in C(\Gw \setminus  K)$ of $Q_{p,A,V}(w)=0$ in $\Gw \setminus K$, we
	have
	$$
	u\leq v \mathrm{~on~} \partial K \Longrightarrow u\leq v \mathrm{~in~} \Gw \setminus K.
	$$
	If $u\in \mathcal{MG}_{A,V,\Gw ; \emptyset},$ then $u$ is called an {\em Agmon ground state} of $Q_{p,A,V}$ in $\Gw.$
\end{definition}
\begin{lemma}[{\cite[Proposition 3.17]{GP1}}]
	Let $V\in M^{q}_{\loc}(p;\Gw)$, and suppose that $Q_{p,A,V} \geq 0$ in $\Gw.$ Then for any $x_0\in \Gw$ the equation $Q_{p,A,V}(w)=0$ admits a unique (up to a multiplicative constant)  solution $u\in \mathcal{MG}_{A,V,\Gw ; \{x_0\}}. $
\end{lemma}
\begin{definition}\label{def:Green}
	A function $u\in \mathcal{MG}_{A,V,\Gw ; \{x_0\}}$ having a nonremovable singularity at $x_0$ is called a {\em minimal positive Green function of $Q_{p,A,V}$ in $\Gw$ with  singularity  at $x_0.$} We denote such a function by $G^{\Gw}_{Q_{p,A,V}}(x,x_0).$
\end{definition}
\begin{lemma}[{\cite[Theorem 5.9]{PP}}]
	Suppose that $Q_{p,A,V} \geq 0$ in $\Gw$. Then $Q_{p,A,V}$ is critical in $\Gw$ if and only if the equation $Q_{p,A,V}=0$ admits a ground state in $\Gw$. 
\end{lemma}
\begin{definition}
	A sequence $\{\phi_k \}_{k\in \N}\subset C_0^{\infty}(\Gw)$ is called a {\em null-sequence}	 with respect to a nonnegative operator $Q_{p,A,V}$ in $\Gw$ if 
	\begin{enumerate}
		\item $\phi_k \geq 0$ for all $k\in \N,$
		\item there exists a fixed open set $B\Subset \Gw$ such that $\| \phi_k\|_{L^p(B)}\asymp 1$ for all $k\in N,$
		\item $\lim\limits_{k\to \infty}\mathcal{Q}^{\Gw}_{p,A,V}(\phi_k)=0.$ 
	\end{enumerate}
\end{definition}
\begin{lemma}[{\cite[Theorem 4.15]{PP}\label{lem:null_seq_non}}]
	A nonnegative operator $Q_{p,A,V}$ is critical in $\Gw$ if and only if $Q_{p,A,V}$ admits a null-sequence in $\Gw.$
\end{lemma}
The next lemma shows that the energy functional $\mathcal{Q}^{\Gw}_{p,A,V}$ is equivalent to a {\em simplified energy} that does not explicitly depend on $V$ and contains only nonnegative terms.
\begin{lemma}[{\cite[Lemma 3.4]{PR}\label{lem:simp_equiv}}]
	Let  $v\in \mathcal{C}^{Q_{p,A,V}}(\Gw)$. Then, for any $0\leq u\in W^{1,p}_{\loc}(\Gw)$ having compact support in $\Gw$, and 
	such that $w:=u/v\in L^{\infty}_{\loc}(\Gw)$, we have 
	\begin{equation}\label{simp_energy}
	\mathcal{Q}^{\Gw}_{p,A,V}(u)\asymp
	\mathcal{Q}^{\Gw}_{\mathrm{sim},p,A,V}(w) 
	:= \int_{\Gw}v^2|\nabla w|_A^2
	\left ( w|\nabla v|_A+v|\nabla w|_A\right )^{p-2} \dx.
	\end{equation}
\end{lemma}
\begin{remark}
	Lemma \ref{lem:simp_equiv} is proved  in \cite{PR} for  the case $V\in L^{\infty}_{\loc}(\Gw)$. However, the proof is purely algebraic and therefore, holds for $V\in M^{q}_{\loc}(p;\Gw)$ as well.
\end{remark}
As a corollary of \eqref{simp_energy} and H\"older's inequality we obtain the following.
\begin{corollary}\label{cor:simp_non}
	Let $v\in \mathcal{C}^{Q_{p,A,V}}(\Gw)\cap C^{1,\alpha}_{\loc}(\Gw)$ and let $X(w):=\int_{\Gw}v^p|\nabla w|_A^p \dx$ and 
	$Y(w):=\int_{\Gw}|w|^p|\nabla v|_A^p \dx$.
	Then, for any continuous function $w\in W^{1,p}(\Gw)$ having compact support in $\Gw$, the following assertions hold true.
	
	1.  $\qquad \mathcal{Q}^{\Gw}_{p,A,V}(vw)\asymp 
	\mathcal{Q}^{\Gw}_{\mathrm{sim},p,A,V}(w)$.  
	$$\!\!\!\!\!\!\!\!\!\!\!\! \!\!\!\!\!\!\!\!
	2. \qquad \mathcal{Q}^{\Gw}_{\mathrm{sim},p,A,V}(w)  \leq
	\begin{cases}
	CX(w)&1<p\leq 2, \\
	C \left[ X(w)+X(w)^{2/p}Y(w)^{\frac{p-2}{p}}\right] & p>2.
	\end{cases}
	$$
\end{corollary}
\subsection{Optimal Hardy-weights}
Let $\overline{\infty}$ denote the ideal point in the one-point compactification of $\Gw$. Let us define the notion of an optimal Hardy-weight  for the operator $Q_{p,A,V} $.
\begin{definition}[{\label{def:optimal}\cite{DP}}]
	Suppose that $Q_{p,A,V} $ is subcritical in $\Gw$. We say that $0\lneqq W$ is an {\em optimal Hardy-weight of $Q_{p,A,V}$ in $\Gw$} if the following two assertions are satisfied:
	\begin{enumerate}
		\item {\bf Criticality:} $Q_{p,A,V-W}$ is critical in $\Gw.$
		\item {\bf Null-criticality  with respect to $W$:} $\int_\Gw |\psi|^p W\dx=\infty$,  where $\psi$ is the (Agmon) ground state of $Q_{p,A,V-W}$ in $\Gw.$
	\end{enumerate}
\end{definition}
\begin{remark}
	Let us discuss Definition \ref{def:optimal}.
	Suppose that $Q_{p,A,V} $ is subcritical in a domain $\Gw$ containing $x_0$, and let $x_0\in K\Subset \Gw$. Then, for any 
	$0\lneqq W \in C_0^{\infty}(\Gw)$ there exists $\tau>0$ such that 
	$Q_{p,A,V-\tau W}$ is critical in $\Gw$ (see for example \cite[Proposition 4.4]{PT}  and \cite{PP}). On the other hand, the ground state of $Q_{p,A,V-\tau W}$, $\phi$, satisfies 
	$$
	\phi\asymp G^{\Gw}_{Q_{p,A,V}}(x,x_0) \qquad \mathrm{in}~
	\Gw \setminus K.
	$$
	Therefore, there are infinity many weight functions $0\lneqq W\in C_0^{\infty}(\Gw)$ such that 
	$Q_{p,A,V -W}$ is critical in $\Gw$, obviously, for such a weight $W$, the operator $Q_{p,A,V -W}$  is not  null-critical with respect to $W$.
\end{remark}
\begin{definition}\label{rem_1.3}
	We say that a Hardy-weight $W$ is {\em optimal at infinity in $\Gw$} if for any  $K\Subset \Gw$, we have 
	$$\sup\{\gl \in \R \mid Q_{p,A,V- \gl W}\geq 0 \mbox{ in } \Gw\setminus K\}=1.$$   	
\end{definition}
\begin{remark}\label{rem_opt_infty}
	The definition of an optimal Hardy-weight in \cite{DFP} includes the requirement that $W$ should be optimal at infinity. But, it is proved  in \cite{KP}  that if $Q - W$ is null-critical with respect to $W$ in $\Gw$, then $Q-W$ is  optimal at infinity. The same proof applies under the assumptions considered in the present paper, hence, in Definition~\ref{def:optimal} we avoid the requirement of optimality at infinity. 
\end{remark}	

The following coarea formula is a direct consequence of \cite[Proposition~3.1]{DP}.
\begin{lemma}{(Coarea formula)}\label{lem:coarea}
	Let $\Gw$ be  a domain in $\R^n$, $n \! \geq \!  2$,   and  $G \! \in \!  C^{1,\alpha}(\Omega)$ is a positive $(p,A)$-harmonic function in $\Gw^* \! := \! \Gw \! \setminus  \! \{0\}$. Assume that for any $0 \! < \! t_1< \! t_2 \!  < \! \infty$, the set
	$\mathcal{A} \! := \! \{x \! \in  \! \Gw^* \mid t_1<G(x)<t_2\}$ is bounded.
	Let $h\in{C^2(0,\infty)} $ be a positive function satisfying $h'(s)>0$ for all $s>0$, and denote $v:=h(G)$.

	Then there exists $C>0$, independent on $t$, such that for any locally bounded real measurable function $f$ such that $f(v)$ has a compact support in $\Gw^*$, we have
	\begin{equation}\label{eq:coarea}
	\int_{\Omega^*}f(v)|\nabla v|_A^p \mathrm{d}x=C\int_{h(\inf_{\Omega*} G)}^{h(\sup_{\Omega*} G)}
	\frac{f(\tau)}{((h^{-1})'(\tau))^{p-1}} \dtau.
	\end{equation}
\end{lemma}
\begin{proof}
	Since $G\in C^{1,\alpha}(\Gw^*)$ and $1<p$, then $\frac{|\nabla G|_A^p}{|\nabla G|}\in L^1_{\loc}(\Gw^*)$ and we may use the (classical) coarea formula (\cite[Theorem 2.32]{CTZ}) to obtain
	for $v=h(G)$
	\begin{align}\label{eq:co_BP}
	&
	\int_{\Gw^*}f(v)|\nabla v|_A^p \dx=
	\int_{\Gw^*}f(h(v))|h'(G)|^p\frac{|\nabla G|_A^p}{|\nabla G|}|\nabla G| \dx= \nonumber\\&
	\int_{\R_+}f(h(t))h'(t)^p\int_{\{G=t\}}\frac{|\nabla G|_A^p}{|\nabla G|} \dH.
	\end{align}
	By (a generalized) Sard's theorem for $C^{1,\alpha}$ functions \cite[Theorem 1.2]{BHS}, $$ 
	\mathcal{H}^{n-1} \left (\{ G=t  \}\cap \mathrm{Crit}(G) \right )=0.
	$$
	The fact that $G\in C^{1,\alpha}$ and Proposition \ref{prop:coarea_f_peri}  imply that (for a.e. $t_1<t_2$) the set 
	$\mathcal{A}:=\{t_1<G<t_2\}$ has a finite perimeter.
	In particular, $\nabla G\neq 0$ and  $\vec{n}$ is well defined on $\partial \mathcal{A}$, $\mathcal{H}^{n-1}$-a.e..
	Let $\partial_+=\{x\in \overline{\mathcal{A}}: G(x)=t_2\}$
	and  $\partial_-=\{x\in \overline{\mathcal{A}}: G(x)=t_1\}$.
	The Gauss-Green theorem (Lemma \ref{lem:Gauss_Green_non})  implies that
	\begin{align*}
	&
	0\!=\!-\!\int_{\mathcal{A}}\!\!\mathrm{div}(|\nabla G|_A^{p-2}A\nabla G) \!\dx\!=\!\!
	\int_{\partial_{+}}+
	\int_{\partial_{-}}\!|\nabla G|_A^{p-2}A\nabla G \!\cdot\! \vec{n} \dH\!= \\&
	\int_{\partial_{+}}|\nabla G|_A^{p-2}A\nabla G \cdot \frac{\nabla G}{|\nabla G|} \dH-
	\int_{\partial_{-}}|\nabla G|_A^{p-2}A\nabla G \cdot \frac{\nabla G}{|\nabla G|} \dH= \\&
	\int_{\{G=t_2\}}\frac{|\nabla G|_A^{p}}{|\nabla G|} \dH-
	\int_{\{G=t_1\}}\frac{|\nabla G|_A^{p}}{|\nabla G|} \dH.
	\end{align*}
	In particular, 
	for any $t>0$,
	$\int_{\{G=t_1\}}\frac{|\nabla G|_A^{p}}{|\nabla G|} \dH=C$. 
	By \eqref{eq:co_BP}, 
	$$
	\int_{\Gw^*}f(v)|\nabla v|_A^p \dx=C\int_{\R_+}f(h(t))h'(t)^p \dt.
	$$
	The change of the variable $h(t)=\tau$ then implies \eqref{eq:coarea}.	
\end{proof}
The following theorem is proved   in \cite{DP} for  the case $A=\mathbf{1}$. However,  it  can be easily checked that  the validity of Lemma \ref{lem:coarea} for a general matrix $A$ satisfying Assumptions \ref{assump2_non_linear}, gives rise to the following theorem. 
\begin{theorem}[{\cite[Theorem~1.5]{DP}}]\label{DPThe}
	Let $\overline{\infty}$ denote the ideal point in the one point compactification of  $\Gw$. Suppose that $-\Delta_{p,A}$ is subcritical in $\Gw$, and admits a positive $(p,A)$-harmonic function $G(x)$ in $\Gw^*:=\Gw\setminus \{0\}$ satisfying one of the following conditions \eqref{eq:a1},\eqref{eq:a2}:
	\begin{equation}\label{eq:a1}
	1<p\leq n, \qquad \lim\limits_{x\to 0}G(x)=\infty, \qquad \mbox{and} \qquad \lim\limits_{x\to \overline{\infty}}G(x)=0,
	\end{equation} 
	\begin{equation}\label{eq:a2}
	p>n, \qquad \lim\limits_{x\to 0}G(x)=\gamma\geq 0, \qquad \mbox{and} \qquad \lim\limits_{x\to \overline{\infty}}G(x)=
	\begin{cases}
	\infty & \mbox{if~} \gamma=0,\\
	0& \mbox{if~} \gamma>0.
	\end{cases}	
	\end{equation}
	Define a positive function $v$ and a nonnegative weight $W$ on $\Gw^*$ as follows:
	\begin{enumerate}
		\item If either \eqref{eq:a1} is satisfied, or \eqref{eq:a2} is satisfied with $\gamma=0$, then 
		$$v:=G^{(p-1)/p},  \; \mbox{ and } \;W:=\left( \frac{p-1}{p}\right)^p\left | \frac{\nabla G}{G}\right |_A^p.$$
		
		\item If \eqref{eq:a2} is satisfied with $\gamma>0$, then 
		$v \! :=  \! [G(\gamma-G)]^{(p-1)/p}$,   and
		$$
		W\! :=\! \left( \!\frac{p-1}{p}\!\right)^p\! \left|\frac{\nabla G}{G(\gamma-G)} \right |_A^p
		\! |\gamma-2G|^{p-2}[2(p-2)G(\gamma-G)+\gamma^2].
		$$
	\end{enumerate}
	Then the following Hardy-type inequality holds in $\Gw^*$:
	\begin{equation}
	\int_{\Gw^*}|\nabla \phi|_A^p \dx\geq\int_{\Gw^*}W|\phi|^p \dx, \qquad \forall\,\phi\in C_0^{\infty}(\Gw^*),
	\end{equation}
	and $W$ is an optimal Hardy-weight of $-\Delta_{p,A}$ in $\Gw^*$.
	Moreover, up to a multiplicative constant, $v$ is the ground state of $-\Delta_{p,A}-W\mathcal{I}_p$ in $\Gw^*$.
\end{theorem}
The following simple observation concerns the existence of optimal Hardy-weights for a `small perturbation' of an operator with  a given optimal Hardy-weight.
\begin{lemma}\label{Vgeq0}
	Assume that   $Q_{p,A,V} $ is subcritical in $\Gw$  and admits an optimal Hardy-weight $W$ in $\Gw^* \! := \! \Gw \! \setminus \!  \{0\}$.
	Let  $V_1\in M^q_\loc(p;\Gw)$ satisfy  $V_1\geq -\vge W$ for some $0\leq \vge <1$ and $q>n/p$.
	Then $W+V_1$ is an optimal Hardy-weight for $Q_{p,A,V+V_1}$ in $\Gw^*$.
\end{lemma}
\begin{proof}
	Consider the function $W \! + \! V_1$.
	Then, $Q_{p,A,V \! + \! V_1} \!  - \! (W \! + \! V_1)\mathcal{I}_p \! = \! Q_{p,A,V} \! - \! W\mathcal{I}_p$  is a critical operator in $\Gw^*$. 
	
	Obviously, $W+V_1\gneqq 0$, and  the ground state $\psi$ of $Q_{p,A,V}-W\mathcal{I}_p$ in $\Gw^*$ is the ground state of $Q_{p,A,V+V_1}-(W+V_1)\mathcal{I}_p$ in $\Gw^*$. Moreover,
	$$
	\int_{\Gw^*}(W+V_1)|\psi|^p \dx \geq (1-\vge)      \int_{\Gw^*}W|\psi|^p \dx= \infty,
	$$ 
	implying that $Q_{p,A,V+V_1-(W+V_1)}$ is null-critical in $\Gw^*$ with respect to $W+V_1$. In particular, $W+V_1$ is an optimal Hardy-weight of $Q_{p,A,V+V_1}$ in $\Gw^*$.
\end{proof}

\section{Optimal Hardy-weights for nonpositive  potentials}\label{sec3} 
Lemma~\ref{Vgeq0} obviously applies when $V_1\geq 0$. The main goal in the current section is to obtain optimal Hardy-weights for a general subcritical operator $Q_{p,A,V}$ in a domain $\Gw$, without assuming $V= 0$ in $\Gw$. In particular,  we prove Theorem \ref{thm:main_non_2}.

First, we recall  the following weak comparison principle \cite[Theorem 5.3]{PP}.
\begin{lemma}[Weak comparison principle]\label{lem:WCP_non}
	Let $\Gw\subset \R^n$ be a bounded Lipschitz domain.
	Assume that $A$ is a uniformly elliptic and bounded matrix in $\Gw$, $V\in M^{q}(p;\Gw)$ and $0\leq g\in L^{\infty}(\Gw)$.
	Assume further that $\lambda_1(\Gw)>0$, where $\lambda_1(\Gw)$ is the principal eigenvalue of the operator $Q_{p,A,V}$. Let $u_2\in W^{1,p}(\Gw)\cap C(\overline{\Gw})$ be a (weak) solution of 
	$$
	\begin{cases}
	Q_{p,A,V}(u_2)= g & \mbox{in~} \Gw, \\
	u_2>0 & \mbox{on ~} \partial \Gw. 
	\end{cases}
	$$
	If $u_1\in W^{1,p}(\Gw)\cap C(\overline{\Gw})$ satisfies 
	$$
	\begin{cases}
	Q_{p,A,V}(u_1)\leq Q_{p,A,V}(u_2) & \mbox{in~} \Gw, \\
	u_1\leq u_2 & \mbox{on ~} \partial \Gw,
	\end{cases}
	$$
	then $u_1\leq u_2$ in $\Gw$.
\end{lemma}

In the following lemma we generalize the notion of Green potential for  $Q_{p,A,V}$.
\begin{lemma}\label{Gphi}
	Assume that $Q_{p,A,V}$ is  subcritical in $\Gw$, and
	let $0\lneqq \vgf\in \core$. Then there exists a   positive function $ G_{\varphi}\in W^{1,p}_{\loc}(\Gw)$,  such that $G_{\vgf}$ is a positive solution of minimal growth at infinity and  satisfies
	$Q_{p,A,V}(G_{\varphi}) =\vgf $ in $\Gw$. 
	
\end{lemma}
\begin{proof}
	Fix $0\lneqq \varphi \in C_0^{\infty}(\Gw)$, and let $\{ \Gw_k\}_{k\in \N}$ be a smooth exhaustion of $\Gw$ with 
	$\mathrm{supp}(\vgf)\Subset \Gw_1$.
	Lemma \ref{lem:see_PP} implies that there exists  a unique positive solution  $G^k\in W^{1,p}(\Gw_k)$ to the problem

	$$
	\begin{cases}
	-\Delta_{p,A}(w)+(V+\frac{1}{k})|w|^{p-2}w=\varphi & \text{~in~} \Gw_k, \\
	w=0& \text{~on~} \partial \Gw_k.
	\end{cases}
	$$
	By the weak comparison principle (Lemma \ref{lem:WCP_non}), $\{ G^k\}_{k\in \N}$ is a monotone increasing sequence of functions. 
	Assume first that  the sequence $\{ G^k\}_{k\in \N}$ is not locally uniformly bounded in $\Gw$, and let $x_1\in \Gw_{2} \setminus \Gw_{1}$.
	By Harnack's convergence principle there exists a subsequence of $\{z_k(x):={G^k(x)}/{G^k(x_1)} \}_{k\in \N}$ which converges locally uniformly to a positive solution $G$, of the equation $Q_{p,A,V}(u)=0$ in $\Gw$.
	Therefore, $G$ is a positive solution of  the equation $Q_{p,A,V}(u)=0$ in $\Gw$ which clearly has minimal growth in a neighborhood of infinity in $\Gw$, i.e., $G$ is a ground state.  This is a contradiction to the subcriticality of the operator $Q_{p,A,V}$ in $\Gw$. \\ 
	Consequently, Harnack inequality (\cite[Theorem 2.7]{PP}) implies that the sequence $\{ G^k\}_{k\in \N}$ is locally uniformly bounded in $\Gw$. By Harnack convergence principle and the strong maximum principle, it converges locally uniformly (up to a subsequence) to a positive solution, $G_{\vgf}$, of the equation  $Q_{p,A,V}(u)=\varphi$ in $\Gw$.
	In fact, \cite[Theorem 5.3]{Lregularity} implies that there exists
	$0<\alpha<1$ such that    $G_{\vgf} \in C^{1,\alpha}_{\loc}(\Gw)$.
\end{proof}  
\begin{definition}\label{def_green}
	Let $0  \! \lneqq   \! \vgf  \! \in   \! \core$. A positive solution 
	$u  \! \in  \!  G_{\vgf}\in \mathcal{MG}_{A,V,\Gw,\supp(\vgf)}$ that satisfies   $Q_{p,A,V}(u)=\vgf$ in $\Gw$,  is called a {\em Green potential}  of $Q_{p,A,V}$  in $\Gw$ with a density $\vgf$.
\end{definition}
We proceed with the following technical proposition (cf. \cite[Lemma 2.10]{DP}).
\begin{proposition}\label{firstprop}
	Let  $f(t)\in C^2(\R_+)$ satisfying  $f,f',-f''>0$. Then, for all $0\leq u\in C^1(\Gw)$
	\begin{equation*} \label{0.5}
	Q_{p,A,V}(f(u))\!=\! 
	-\Delta_p^{1D}(f)(u)|\nabla u|^p_A+(f'(u))^{p-1} \!  \! \left(\!\! -\Delta_{p,A}(u)  \! +  \!  V\left(\frac{f(u)}{f'(u) u} \!\right)^{  \! p-1}  \! |u|^{p-1}  \! \right) 
	\end{equation*}
	in the the weak sense.
	Here $-\Delta_p^{1D} f(t):=-(|f'(t)|^{p-2}f'(t))'$  is the one-dimensional $p$-Laplacian.
\end{proposition}
\begin{proof}
	\
	By \cite[Lemma 2.10]{DP} (which clearly holds for the $(p,A)$-Laplacian), we have:
	\begin{equation}\label{eq_lap_fu} 
	-\Delta_{p,A}(f(u))=-|f'(u)|^{p-2}\left[(p-1)f''(u)|\nabla u|_A^p+f'(u)\Delta_{p,A}(u)\right]
	\end{equation}
	in the weak sense.
	Since $f\in C^2,f,f',-f''>0$ we have
	$$
	-|f'(u)|^{p-2}  (p-1)f''(u)|\nabla u|_A^p  =-\frac{d}{dt}[|f'(t)|^{p-1}](u)|\nabla u|_{A}^p=-\Delta_p^{1D}(f)(u)|\nabla u|_A^p \,,
	$$
	and together with  \eqref{eq_lap_fu} the proposition is proved.
\end{proof}
\begin{remark}
	We remark that if $f(t)=t^{\frac{p-1}{p}}$,  then 
	$$
	-\Delta_p^{1D}(f(t))- \! \!\left(\frac{p-1}{p}\right)^p \! \frac{f(t)^{p-1}}{t^{p}} \!=\!0, \;\; \mbox{and }  c_p:= \! \left(\frac{f(u)}{f'(u) u}\right)^{p-1}\!\!=\!\left(\frac{p}{p-1}\right)^{p-1}\!\!>1.
	$$ 	
\end{remark}
Lemma \ref{Gphi} and Proposition \ref{firstprop} imply:
\begin{corollary}
	Assume that $Q_{p,A,c_p V}$ is subcritical in $\Gw$. For $0\lneq\vgf\in \core$, let $G_{\vgf}$  be a Green potential satisfying $Q_{p,A,c_p V}(G_{\vgf})=\vgf$ in $\Gw$, and let $f(t)=t^{\frac{p-1}{p}}$. 
	Then,
	\begin{equation}
	Q_{p,A,V}(f(G_{\vgf}))= 
	-\Delta_p^{1D}(f)(G_{\varphi})|\nabla G_{\varphi}|_A^p+(f'(G_{\varphi}))^{p-1}\varphi \gneqq 0.
	\end{equation}
	In particular, $f(G_{\varphi})$ is a positive solution of the equation
	$Q_{p,A,V-W}(v) =0$, where 
	$$
	W=\frac{Q_{p,A,V}(f(G_{\varphi}))}{f(G_{\varphi})^{p-1}}, \quad \mbox{and } \quad 
	W=\left( \frac{p-1}{p}\right )^p\left|\frac{\nabla G_{\varphi}}{G_{\varphi}}\right|_A^p \qquad \text{in}~ \Gw \setminus \mathrm{supp}(\varphi).
	$$ 
\end{corollary}

The following lemma is a generalization of Lemma \ref{lem:coarea} to the case $V\neq 0$.
\begin{lemma}\label{coarea grad G}
	Assume that  $Q_{p,A,V}$ is subcritical in $\Gw$,  and let
	$G_{\varphi}\in C^{1,\alpha}_{\loc}(\Gw)$ be a Green potential (with respect to $0\lneq \vgf\in \core$), and assume that 
	\begin{equation}\label{eq:assum_G_G_V}
	\lim\limits _{x\to \overline{\infty}}G_{\varphi}=0;	\qquad  \quad 		\int_{\Gw}VG_{\vgf}^{p-1} \dx< 0;
	\quad \quad		\int_{\Gw}|V||G_{\vgf}|^{p-1} \dx<\infty.
	\end{equation}
	
	Then, there exists $0<M_{\vgf}<\sup\limits_{\Gw}G_{\varphi}$ such that 
	for almost every  $0<t<M_{\vgf}$, satisfying
	$$\text{supp}(\varphi)\Subset \Gw_t:=\{x\in \Gw:G_{\vgf}(x)>t \},$$
	there exists $C>0$, independent of $t$, such that 
	\begin{equation}\label{eq:coarea_modified}
	C^{-1}\leq \int_{G_{\varphi}=t} |\nabla G_{\varphi}|_A^{p-1}\,\mathrm{d} \sigma_A  \leq C,
	\end{equation}
	where $\mathrm{d} \sigma_A=\frac{|\nabla G_{\vgf}|_A}{|\nabla G_{\vgf}|}\, \dH$, $\mathcal{H}^{n-1}$-a.e.
\end{lemma}
\begin{proof}
	\
	The assumption $\lim\limits _{x\to \overline{\infty}}G_{\varphi}=0$, and Proposition \ref{prop:coarea_f_peri} imply that for a.e. $t>0$   the set 
	$\Gw_t$ has finite perimeter. 
	Furthermore, \eqref{eq:assum_G_G_V} implies that 
	$|V|G_{\vgf}^{p-1}\in \mathcal{M}(\Gw')$. Finally, Sard's theorem for $C^{1,\alpha}$-functions implies that the conditions in Gauss-Green theorem (Lemma~\ref{lem:Gauss_Green_non} are satisfied in $\Gw'$. Hence, 
	\begin{equation*}
	\! \int_{\Gw_t}\!\!( \varphi  \! -  \! V|G_{\varphi}|^{p-2}G_{\vgf}\! ) \!\dx\!=\!-  \!   \! \int_{\Gw_t}\!  \! 
	\! \!\diver  \! (|\nabla G_{\varphi}|_A^{p-2}  \!   \! A\nabla G_{\varphi}) \!\dx\!=\! 
	-  \!   \! \int_{\partial \Gw_t}      \! \! \!  \! |\nabla G_{\varphi}|_A^{p-2}  \! A\nabla G_{\varphi} \cdot  \vec{n}\! \dH.
	\end{equation*}
	The assumptions 	$\lim\limits _{x\to \overline{\infty}}G_{\varphi}=0,$ and	$ 		\int_{\Gw}VG_{\vgf}^{p-1} \dx< 0$ imply that for a sufficiently small $M_{\vgf}>0$ and $0<t<M_{\vgf}$,
	\begin{equation}\label{eq:M} 
	 \int_{\Gw_t}\big (\varphi-V|G_{\varphi}|^{p-2}G_{\vgf} \big )\dx \leq 
	\int_{\Gw}\big (\vgf+ |V||G_{\vgf}|^{p-1} \big ) \dx\leq C.
	\end{equation}

	Moreover, the assumption $\text{supp}(\vgf)\Subset\Gw_t$ implies 
	$$
	C^{-1} \leq \int_{\Gw} \vgf \dx= \int_{\Gw_t} \vgf \dx\leq\int_{\Gw_t}\big (\varphi-V|G_{\varphi}|^{p-2}G_{\vgf} \big )\dx.
	$$
	Consequently,
	$$
	\int_{\Gw_t}\big (\vgf-V|G_{\vgf}|^{p-2}G_{\vgf}\big ) \dx \asymp C,
	$$
	and $C$ does not depend on $t$.
	Sard's theorem for $C^{1,\alpha}$ functions implies that for $\mathcal{H}^{n-1}$-a.e. $x\in \partial \Gw'$, $|\nabla G(x)|\neq 0$. Furthermore,  the definition of $\Gw'$ implies that $G_{\vgf}\geq t$ in $ \Gw'$, and hence, $\vec{n}=-\frac{\nabla G_{\vgf}}{|\nabla G_{\vgf}|}$ for $\mathcal{H}^{n-1}$-a.e. $x\in \partial \Gw'$. Therefore,
	$$
	-\int_{\partial \Gw_t}|\nabla G_{\varphi}|_A^{p-2}A\nabla G_{\varphi}\cdot \vec{n}\dH =
	\int_{\partial \Gw_t}|\nabla G_{\varphi}|_A^{p-1}\frac{|\nabla G_{\vgf}|_A}{|\nabla G_{\vgf}|}\dH \asymp C. \quad \qedhere
	$$ 
\end{proof}
\begin{remark}\label{rem:M}
	\em{
	The assumption $\int_{\Gw}VG_{\vgf}^{p-1} \dx< 0$ in Lemma \ref{coarea grad G} is needed for arguing \eqref{eq:M}. In particular, the lemma  still holds once assuming instead that $V\leq 0$ in $\Gw$.
	}
\end{remark}
We proceed with the following lemma. 
\begin{lemma}[{cf. \cite[Propositions 5.1 and 5.5]{DP}}]\label{optmal Lp}
	Let  $0\lneqq \varphi\in \core$, and assume that $Q_{p,A,c_p V}$ is subcritical in $\Gw$. Let
	$G_{\varphi}\in C^{1,\alpha}_{\loc}(\Gw)$ be a Green potential satisfying $Q_{p,A,c_pV}(G_{\vgf})=\varphi \quad \text{in } \Gw$, and assume that
	$$
	\lim\limits _{x\to \overline{\infty}}G_{\varphi}=0;	\qquad  \quad 		\int_{\Gw}VG_{\vgf}^{p-1} \dx< 0;
	\quad \quad		\int_{\Gw}|V||G_{\vgf}|^{p-1} \dx<\infty.
	$$
	Consider the function $f(t)=t^{\frac{p-1}{p}},$ and let 
	$$
	W:=\frac{Q_{p,A,V}(f(G_{\varphi}))}{f(G_{\varphi})^{p-1}}\, .
	$$
	Then $Q_{p,A,V-W}$ is critical in $\Gw,$ with a ground state $f(G_{\varphi})$ and   $\int_{\Gw}Wf(G_{\varphi})^p \dx=\infty$.
	Hence, $W$ is an optimal Hardy-weight  for  $Q_{p,A,V}$ in $\Gw.$ 
\end{lemma}
\begin{proof}
	\
	{\bf Criticality:} 
 	Let $M_{\vgf}$ be given by Lemma \ref{coarea grad G}, and let $K\Subset \Gw$ be a precompact smooth subdomain satisfying
	$\text{supp} \, \varphi \Subset K$, $\max \limits_{\Gw\setminus K}G_{\vgf}<M_{\vgf}$
	and $G_{\varphi }<1$  for all $x\in \Gw \setminus K$.
	Assume without loss of generality that $\inf\limits_{K}G_{\varphi}\geq 1$.
	
	For each $k\in \N$, consider the function $\phi_k(f(G_{\varphi}))$, where $f(t)=t^{\frac{p-1}{p}}$ and 
	\begin{equation*} \label{eq:phi_k}
	\phi_k(t)=
	\begin{cases}
	0 & 0 \leq t\leq \frac{1}{k^2},\\
	2+\frac{\log t}{\log k} & \frac{1}{k^2} \leq t\leq \frac{1}{k},\\
	1 & \frac{1}{k}\leq t \leq k,\\
	2-\frac{\log t}{\log k} & k \leq t\leq k^2,\\
	0 & t \geq k^2.
	\end{cases}
	\end{equation*}
	We claim that $u_k= \phi_k(f(G_{\varphi}))f(G_{\varphi})$ is a null-sequence of $Q_{p,A,V-W}$ in $\Gw$. Indeed, by \eqref{simp_energy}, $\mathcal{Q}_{\mathrm{sim}}(w)\asymp 
	\mathcal{Q}(wf(G_{\varphi})) = \mathcal{Q}(u)$,	where 
	$$\mathcal{Q}(u)
	= \int_{\Gw}\big(|\nabla u|_A^p+(V-W)|u|^p \big) \dx,$$  
	and
	$$\mathcal{Q}_{\mathrm{sim}}(w) = \int_{\Gw} f(G_{\varphi})^2|\nabla w|_A^2 \left ( w|\nabla (f(G_{\varphi}))|_A + f(G_{\varphi}) |\nabla w|_A\right )^{p-2}\!\! \dx.$$
	Moreover,  by Corollary~\ref{cor:simp_non}  we have 
	\begin{equation*}\label{eq:X_simp_non_lin}
	\mathcal{Q}_{\text{sim}}(w)\leq
	\begin{cases}
	CX(w)&1<p\leq 2, \\
	C \left[ X(w)+X(w)^{2/p}Y(w)^{\frac{p-2}{p}}\right] & p>2,
	\end{cases}
	\end{equation*}
	where 
	$$
	X(w)=\int_{\Gw}|\nabla w|_A^p f(G_{\varphi})^p \dx, \qquad Y(w)=\int_{\Gw} |w|^p|\nabla (f(G_{\varphi}))|_A^p \dx.
	$$
	By the (classical) coarea formula (\cite[Theorem 2.32]{CTZ}),
	\begin{align*}&
	X(\phi_k(f(G_{\varphi})))=\int_{\Gw\setminus K} f(G_{\varphi})^p|\phi_k'(f(G_{\varphi}))|^p|f'(G_{\varphi})|^p|\nabla G_{\varphi}|_A^p \dx= \\ &
	\int_{0}^{\max\limits_{\Gw\setminus K}G_{\varphi}}f(t)^p|\phi_k'(f(t))|^pf'(t)^p\dt\int_{G_{\varphi}=t} |\nabla G_{\varphi}|_A^{p-1}\text{d} \sigma_A.
	\end{align*}
	By Lemma \ref{coarea grad G}, for a.e. $0<t<\max\limits_{\Gw\setminus K}G_{\varphi}$ we have  $\int_{G_{\varphi}=t} |\nabla G_{\varphi}|_A^{p-1}\text{d} \sigma_A  \asymp 1$. Moreover,
	\begin{align*}
	\int_{0}^{\max\limits_{\Gw\setminus K}G_{\varphi}}  f(t)^p|\phi_k'(f(t))|^pf'(t)^p\!\dt=& 
	C(p)\int_0^{f(\max\limits_{\Gw\setminus K}G_{\varphi})} \frac{|s\phi_k'(s)|^p}{s}\!\ds=\\
	&\frac{C(p)}{\log^p \!k}\int_{\frac1{k^2}}^{\frac1k}\!\frac{1}{s} \ds\asymp \left(\!\frac{1}{\log k} \!\right)^{\!\!p-1}.
	\end{align*} 
	Consequently, $X(\phi_n(f(G_{\varphi})))\asymp \left (\dfrac{1}{\log k} \right )^{p-1}$. 
	By a similar calculation, 
	\begin{align*}&
	Y(\phi_k(f(G_{\varphi})))=\int_{\Gw\setminus K}|\phi_k(f(G_{\varphi}))|^pf'(G_{\varphi})^p|\nabla G_{\varphi}|_A^p \dx\asymp \int_{0}^{1}|\phi_k(f(t))|^pf'(t)^p 
	\dt \asymp\\ &
	\int_{0}^{f(1)}|\phi_k(s)|^p\frac{ds}{s}=
	\int_{1/k^2}^{1/k}\left ( 2+\frac{\log s}{\log k}\right )\frac{1}{s} \ds+\int_{1/k}^1\frac{1}{s}\ds
	\asymp \int_{1/k}^{1}\frac{1}{s} \ds\asymp \log k . 
	\end{align*}
	It follows that
	$\mathcal{Q}_{\mathrm{sim}}(w_k)=\mathcal{Q}_{\mathrm{sim}}(\phi_k(f(G_{\varphi})))\to 0$ as $k \to \infty$, and therefore,
	$$\mathcal{Q}(u_k)=\mathcal{Q}(\phi_k(f(G_{\varphi})f(G_{\varphi})))\to 0 ~ \text{as}~ k\to \infty.$$ 
	Let us specialize $\varepsilon_0>0$ such  that  the set $B=\{x\in \Gw: \varepsilon_0/2<f(G_{\varphi})<\varepsilon_0 \}$ is nonempty, bounded,
	and contained in $\Gw \setminus K.$
	Therefore, 
	\begin{equation}\label{eq:same_phi_k}
	\int_{B} |u_k|^p \dx=\int_{B}|\phi_k(f(G_{\varphi}))|^p f(G_{\varphi})^p\dx\asymp 1.
	\end{equation}
	Thus, the sequence $\{u_k\}$ is a null-sequence, and in light of Lemma \ref{lem:null_seq_non}, $Q_{p,A,V-W}$ is critical in  
	$\Gw$.
	
	\noindent {\bf Null-criticality:}
	Let $K\Subset \Gw$ be a precompact smooth subdomain as in the first part of the proof.
	
	For almost every $0<\tau<1$ we consider the set $\Gw_{\tau}:=\{x\in \Gw \mid  \tau<G_{\varphi}<\min\limits_{K}G_{\varphi} \}$
	which has finite perimeter.
	Recall that
	$$
	W=\left( \frac{p-1}{p}\right )^p\frac{|\nabla G_{\vgf}|_A^p}{G_{\vgf}^p} \qquad \mbox{in } \Gw_{\xi}.
	$$
	By the (classical) coarea formula  and \eqref{eq:coarea_modified},
	\begin{align*}&
	\int _{\Gw_{\tau}}W(f(G_{\varphi}))^p \dx= \left( \frac{p-1}{p}\right )^p\int _{\Gw_{\tau}}\frac{|\nabla G_{\vgf}|_A^p}{G_{\vgf}^p}(f(G_{\varphi}))^p \dx=\\[2mm] &
	\left( \frac{p-1}{p}\right )^p\int_{\R_+}\left (\frac {f(t)}{t} \right )^p \dt \int _{G_{\varphi=t}}|\nabla G_{\vgf}|_A^{p-1} \text{d}\sigma_A
	\asymp C \int_{\tau}^{\min \limits_K G_{\varphi}} \left (\frac {f(t)}{t} \right )^p \dt.
	\end{align*}
	By letting $\tau \to 0$ we obtain that $\int_{\Gw \setminus K}Wf(G_{\varphi})^p \dx=\infty.$
\end{proof}
\begin{remark}\label{rem:M2}
	\em{
	Remark \ref{rem:M} implies that Lemma \ref{optmal Lp} still holds if one assumes $V\leq 0$ in $\Gw$ instead of the assumption
	$\int_{\Gw}VG_{\vgf}^{p-1} \dx< 0$. 
	}
\end{remark}

\begin{proof}[Proof of Theorem \ref{thm:main_non_2}]
	Notice that $c_p > 1$, and hence  $Q_{p,A,{V}/{c_p}}$ is subcritical in $\Gw$.
	Let $G_{\vgf}$ be the Green potential of $Q_{p,A,V}$, given by Lemma \ref{Gphi}.
	By Lemma \ref{optmal Lp}, the operator  $Q_{p,A,{V}/{c_p}}$ admits an optimal Hardy-weight in $\Gw$.
\end{proof} 
\begin{proof}[Proof of Corollary \ref{thm:main_non_1}]
	\
	Notice that $c_p > 1$, and hence  $Q_{p,A,{V}/{c_p}}$ is subcritical in $\Gw$.
	Let $G_{\vgf}$ be the Green potential of $Q_{p,A,V}$, given by Lemma \ref{Gphi}.
	By the minimal growth property of $G_{\vgf}$, for any $x_0\in K\Subset \Gw$, $G_{\vgf}\leq CG$ in $\Gw\setminus K$, and in particular, 
	$$
	\lim_{x\to \overline{\infty}}G_{\vgf}=0, \qquad \int_{\Gw}|V||G_{\vgf}|^{p-1} \dx<\infty.
	$$
	By   Lemma \ref{optmal Lp} and Remark \ref{rem:M2}, the operator  $Q_{p,A,{V}/{c_p}}$ admits an optimal Hardy-weight in $\Gw$.
	
\end{proof}
Corollary~\ref{thm:main_non_1} and the following remark give rise to new optimal Hardy-type inequalities in the smooth case. 
\begin{remark}\label{Rem:1}
	Let $\Gw\subset \R^n$ be a domain and  let
	$Q_{p,A,V}$ be a subcritical operator in $\Gw$ satisfying Assumptions \ref{assump2_non_linear}. 
	Assume further that $V\leq 0$ in $\Gw$. Then, there exists  $K\Subset \Gw$  and $x_0\in  \mathrm{int} K \Subset \Gw$, such that the operator $Q_{p,A,V}$ admits a positive solution $G(x)$ in $\Gw\setminus \{x_0\}$ satisfying \eqref{eq:Gphi_min_non}
	in each of the following cases :
	\begin{itemize}
		\item $A$ is a constant, symmetric,  positive definite matrix; $V\in L^{\infty}(\Gw)$; $\Gw$ is a bounded $C^{1,\alpha}$ domain and $\lambda_1(\Gw)>0$ \cite{L}.
		\item$A$ is a constant, symmetric,  positive definite matrix; $V\in C_0^{\infty}(\R^n)$; $\Gw=\R^n$ \cite{FP1,GP1}.  
	\end{itemize}
	In particular, Theorem \ref{thm:main_non_2} can be applied in each of the latter cases. 
\end{remark}
\begin{remark}\label{Rem:6}
	Combining Theorem~\ref{thm:main_non_2} and Lemma~\ref{Vgeq0}, we obtain optimal Hardy-weights for a wide family of operators $Q_{p,A,V}$ with indefinite potentials $V$.
\end{remark}
\section*{Acknowledgments}
The paper is based on part of the author's Ph.~D. thesis. The author expresses his gratitude to his thesis adviser, Professor Yehuda Pinchover,
for the encouragement, support and help he gave him. The author is grateful to the Technion for supporting his study.
\bibliographystyle{amsplain}

\begin{thebibliography}{10}
	\bibitem{Agmon}
	S.~Agmon, "Lectures on Exponential Decay of Solutions of Second-Order Elliptic
	Equations: Bounds on Eigenfunctions of N-body Schr\"odinger Operators", Mathematical Notes, 29, Princeton University Press, Princeton, 1982.
	
	
	\bibitem{BEL} A.A.~Balinsky, W.D.~Evans, and R.T.~Lewis: ``The Analysis and Geometry of Hardy's Inequality", Universitext. Springer, Cham, 2015.
	
	\bibitem{BFT}	G.~Barbatis, S.~Filippas, and A.~Tertikas, 
	Sharp Hardy and Hardy-Sobolev inequalities with point singularities on the boundary,\textit{ J. Math. Pures Appl.}  \textbf{117} (2018), 146--184.
	\bibitem{BD}
	K.~Bogdan, and B.~Dyda, 
	The best constant in a fractional Hardy inequality,
	{\em Math. Nachr.} {\bf 284} (2011),  629--638.
	%
	\bibitem{BHS}
	B.~Bojarski, P.~Hajlasz, and P.~Strzelecki, 
	Sard's theorem for mappings in H\"older and Sobolev spaces,
	{\em Manuscripta Math.} {\bf 118} (2005), 383--397.
	%
	%
	\bibitem{BM} H.~Brezis, and M.~Marcus, Hardy's inequalities revisited, {\em Ann. Scuola Norm. Sup. Pisa Cl. Sci.} (4) {\bf 25} (1997),  217--237. 
	%
	%
	\bibitem{CTZ}
	G.~Chen, M.~Torres, and W.~Ziemer,
	Gauss-Green theorem for weakly differentiable vector fields, sets of finite perimeter, and balance laws,
	{\em Comm. Pure Appl. Math.} {\bf 62} (2009), 242--304.
	\bibitem{DFP} B.~Devyver, M.~Fraas, and Y.~
	Pinchover, Optimal Hardy weight for second-order elliptic operator: an answer to a problem of Agmon, {\em J. Funct. Anal.} {\bf 266} (2014), 4422--4489.
	
	\bibitem{DP}
	B.~Devyver, and Y.~Pinchover, Optimal $L^p$ Hardy-type inequalities, {\em Ann. Inst. H. Poincar\'e Anal. Non Lin\'eaire} {\bf 33} (2016), 93--118.
	\bibitem{EHL}
	T.~Ekholm, H.~Kova\v{r}\'{\i}k, and A.~Laptev, Hardy inequalities for $p$-Laplacians with Robin boundary conditions, {\em Nonlinear Anal.} {\bf 128} (2015), 365--379.
	%
	\bibitem{EG} L.C.~Evans, and R.F.~Gariepy, ``Measure Theory and Fine Properties of Functions", Revised edition, Textbooks in Mathematics. CRC Press, Boca Raton, FL, 2015. 
	%
	
	\bibitem{FP1}
	M.~Fraas, and Y.~Pinchover, Isolated singularities of positive solutions of $p$-Laplacian type equations in $\R^d$, {\em J. Differential Equations} {\bf 254} (2013), 1097--1119.
	\bibitem{GP1}
	R.~Giri , and Y.~Pinchover, Positive Liouville theorem and asymptotic behaviour for $(p,A)$-Laplacian type elliptic equations with Fuchsian potentials in Morrey space,
	 {\em in} a topical collection: Harmonic Analysis and PDE
	  dedicated to the 80th birthday of  Vladimir Maz'ya,  Anal. Math. Phys. 10, Article number: 67 (2020).
		\bibitem{KPP} M.~Keller, Y.~Pinchover, and F.~Pogorzelski, Optimal Hardy inequalities for Schr\"odinger operators on graphs, {\em Comm. Math. Phys.} {\bf 358} (2018), 767--790.
	%
	%
	
	\bibitem{KP} H.~Kova\v{r}\'{\i}k, and Y.~Pinchover, On minimal decay at infinity of Hardy-weights, {\em Comm. Contemp. Math.} {\bf 22} (2019), 1950046. 
	%
	\bibitem{Lregularity}
	G.M.~Lieberman,  Sharp form of estimates for subsolutions and supersolutions of quasi-linear elliptic
	equations involving measures,{\em Comm. Partial Differential Equations} {\bf 18} (1993), 1191--1212.
	
	%
	\bibitem{L} G.M.~Lieberman, ``Oblique Derivative Problems for Elliptic Equations", World Scientific Publishing Co, Pte. Ltd., Hackensack, NJ, 2013.
	%
	\bibitem{MZ} J.~Mal\'{y}, and W.~P.~Ziemer, ``Fine Regularity of Solutions of Elliptic Partial Differential Equations", Mathematical Surveys and Monographs 51, American Mathematical Society, Providence, RI, 1997.
	%
	\bibitem{PP}
	Y.~Pinchover, and G.~Psaradakis, On positive solutions of the $(p,A)$-Laplacian with potential in Morrey space, {\em Anal. PDE} {\bf 9} (2016), 1317--1358.
	%
	\bibitem{PR}
	Y.~Pinchover, and N.~Regev,	Criticality theory of half-linear equations with the $(p,A)$-Laplacian, {\em Nonlinear Anal.} {\bf 119} (2015), 295--314.
	\bibitem{PT}
	Y. Pinchover and K. Tintarev, Ground state alternative for $p$-Laplacian with potential term, {\em Calc. Var. Partial Differential Equations} {\bf 28} (2007), 179--201
	
	\bibitem{PV} Y.~Pinchover, and I.~Versano, Optimal Hardy-weights for elliptic operators with mixed boundary conditions, arXiv: 2103.13979.
	
	
\end{thebibliography}

\end{document}